\documentclass[12pt,reqno]{amsart}
\usepackage{amssymb}
\usepackage{amsxtra}
\usepackage{mathrsfs}
\usepackage[all]{xy}
\usepackage{cite}
\usepackage{paralist}
\usepackage[hypertex]{hyperref}
\newtheorem{theorem}{Theorem}[section]
\newtheorem{lemma}[theorem]{Lemma}
\theoremstyle{definition}
\newtheorem{definition}[theorem]{Definition}
\theoremstyle{remark}
\newtheorem{remark}[theorem]{Remark}
\DeclareMathOperator{\Hom}{Hom}
\DeclareMathOperator{\RHom}{RHom}
\DeclareMathOperator{\Coker}{Coker}
\DeclareMathOperator{\Tor}{Tor}
\DeclareMathOperator{\Ext}{Ext}
\DeclareMathOperator{\Ob}{Ob}
\newcommand*{\dL}{\mathrm{L}}

\newcommand*{\dD}{\mathrm{D}}
\newcommand*{\lmod}{\mbox{-}\!\mathop{\mathsf{mod}}}
\newcommand*{\rmod}{\mathop{\mathsf{mod}}\!\mbox{-}}
\newcommand*{\bimod}{\mbox{-}\!\mathop{\mathsf{mod}}\!\mbox{-}}
\newcommand*{\LCS}{\mathsf{LCS}}
\newcommand*{\Vect}{\mathsf{Vect}}
\newcommand*{\Ptens}{\mathop{\widehat\otimes}}
\newcommand*{\ptens}[1]{\mathop{\widehat\otimes}_{#1}}
\newcommand*{\Lptens}[1]{\mathop{\widehat\otimes}_{#1}\nolimits^{\dL}}
\newcommand*{\holim}{\mathop{\underaccent{\longleftarrow}{\mathrm{holim}}}}
\newcommand*{\id}{1}
\newcommand*{\N}{\mathbb N}
\newcommand*{\CC}{\mathbb C}
\newcommand*{\Z}{\mathbb Z}
\newcommand*{\cA}{\mathscr A}
\newcommand*{\cB}{\mathscr B}
\newcommand*{\cT}{\mathcal T}
\newcommand*{\bC}{\mathsf C}
\newcommand*{\bH}{\mathsf H}
\newcommand*{\bD}{\mathsf D}
\newcommand*{\op}{\mathrm{op}}
\newcommand*{\eps}{\varepsilon}
\newenvironment{mycompactenum}{\pltopsep=5pt\begin{compactenum}[\upshape (i)]}%
{\end{compactenum}}
\newcommand*{\xra}{\xrightarrow}
\makeatletter
\def\varholim@#1#2{%
  \vtop{\m@th\ialign{##\cr
    \hfil$#1\operator@font holim$\hfil\cr
    \noalign{\nointerlineskip\kern1.5\ex@}#2\cr
    \noalign{\nointerlineskip\kern-\ex@}\cr}}%
}
\def\hocolim{%
  \mathop{\mathpalette\varholim@{\rightarrowfill@\textstyle}}\nmlimits@
}
\def\holim{%
  \mathop{\mathpalette\varholim@{\leftarrowfill@\textstyle}}\nmlimits@
}
\makeatother
\numberwithin{equation}{section}
\begin{document}
\title[Relative homological epimorphisms]{A note on relative homological epimorphisms\\[2pt]
of topological algebras}
\subjclass[2010]{46M15, 46M18, 16E30, 16E35, 46H99}
\author{A. Yu. Pirkovskii}
\address{Alexei Yu. Pirkovskii, Faculty of Mathematics,
HSE University,
6 Usacheva, 119048 Moscow, Russia}
\email{aupirkovskii@hse.ru}
\thanks{This work was supported by the RFBR grant no. 19-01-00447.}
\date{}

\begin{abstract}
We give several characterizations of relative homological epimorphisms in the setting of
locally convex topological algebras, thereby correcting
a gap in our earlier paper [Trans. Moscow Math. Soc. 2008, 27--104].
\end{abstract}

\maketitle

\section{Introduction}

Homological epimorphisms of topological algebras
were introduced by J.~L.~Taylor \cite{T2} under the name of {\em absolute localizations}.
Since then, they were rediscovered several times under different names
(see \cite{Dicks,GL,NR,Meyer,BBK}), both in the purely algebraic and in the functional analytic
contexts. A more detailed historical survey is given in \cite[Remark 3.16]{AP}.

In \cite{Pir_qfree}, we introduced a relative version of this notion, which roughly
corresponds to the situation where homological properties of topological $A$-modules are
considered relative to a subalgebra $R$ of $A$. In other words, we worked in the
functional analytic version of G.~Hochschild's relative homological algebra \cite{Hoch_rel}.
Our main motivation was that, in order to prove that an algebra homomorphism
$\varphi\colon A\to B$ is a homological epimorphism, it is often convenient
(and it often suffices, provided that some natural conditions are satisfied) to construct finite chains
$R_0\subset R_1\subset\cdots\subset R_n=A$ and
$S_0\subset S_1\subset\cdots\subset S_n=B$ of subalgebras such that $\varphi(R_i)\subset S_i$
for all $i$, and such that $\varphi|_{R_i}\colon R_i\to S_i$ is a homological epimorphism
relative to $R_{i-1}\subset R_i$ and $S_{i-1}\subset S_i$, respectively.
This method applies, in particular, to iterated Ore extensions and to their functional
analytic counterparts (see \cite[Section 9]{Pir_qfree}).

The goal of this note is to correct a gap in \cite[Theorem 6.3]{Pir_qfree}, where
several characterizations of relative homological epimorphisms were given.
For the non-relative case, such characterizations are known in the purely algebraic setting
\cite{GL,Pauk,NS} and in the setting of bornological algebras \cite{Meyer}.
However, the setting of locally convex topological algebras requires some different tools,
mostly because there is no internal $\Hom$ functor in the category of complete locally
convex spaces, and so the adjunctions between $\Hom$ and $\otimes$ are not available
in the context of complete locally convex modules.

\section{Preliminaries}

Throughout, all vector spaces and algebras are assumed to be over the field $\CC$
of complex numbers. All algebras are assumed to be associative and unital.
By a {\em $\Ptens$-algebra} we mean an algebra $A$ endowed with
a complete locally convex topology in such a way that the multiplication
$A\times A\to A$ is jointly continuous.
Note that the multiplication uniquely extends to a continuous linear map
$A\Ptens A\to A$, $a\otimes b\mapsto ab$,
where the symbol $\Ptens$ stands for the completed projective
tensor product (whence the name ``$\Ptens$-algebra'').
If $A$ is a $\Ptens$-algebra, then a {\em left $A$-$\Ptens$-module}
is a left $A$-module $X$ endowed with
a complete locally convex topology in such a way that
the action $A\times X\to X$ is jointly continuous.
We always assume that $1_A\cdot x=x$ for all $x\in X$, where $1_A$ is the identity of $A$.
Left $A$-$\Ptens$-modules and their continuous morphisms form a category
denoted by $A\lmod$.
The categories $\rmod A$ and $A\bimod B$ of right
$A$-$\Ptens$-modules and of $A$-$B$-$\Ptens$-bimodules are defined similarly.
Given left $A$-$\Ptens$-modules $X$ and $Y$, we denote by $\Hom_A(X,Y)$
the space of morphisms from $X$ to $Y$.
If $X$ is a right $A$-$\Ptens$-module and $Y$
is a left $A$-$\Ptens$-module, then their {\em $A$-module tensor product}
$X\ptens{A}Y$ is defined to be
the completion of the quotient $(X\Ptens Y)/L$, where $L\subset X\Ptens Y$
is the closed linear span of all elements of the form
$x\cdot a\otimes y-x\otimes a\cdot y$
($x\in X$, $y\in Y$, $a\in A$).
As in pure algebra, the $A$-module tensor product can be characterized
by the universal property that, for each complete locally convex space $E$,
there is a natural bijection between the set of all
continuous $A$-balanced bilinear maps from $X\times Y$ to $E$
and the set of all continuous linear maps from
$X\ptens{A}Y$ to $E$.

If $R$ is a $\Ptens$-algebra, then an {\em $R$-$\Ptens$-algebra} is a pair $(A,\eta_A)$
consisting of a $\Ptens$-algebra $A$ and a continuous algebra homomorphism
$\eta_A\colon R\to A$. When speaking of $R$-$\Ptens$-algebras we often abuse the language
by saying that $A$ is an $R$-$\Ptens$-algebra (without explicitly mentioning $\eta_A$);
this should not lead to confusion. If $A$ is an $R$-$\Ptens$-algebra and $B$
is an $S$-$\Ptens$-algebra, then an {\em $R$-$S$-homomorphism} from $A$ to $B$
is a pair $(f,g)$, where $f\colon A\to B$ and $g\colon R\to S$ are continuous algebra
homomorphisms such that $f\eta_A=\eta_B g$. Again, we often abuse the language by
saying that $f\colon A\to B$ is an $R$-$S$-homomorphism (without explicitly
mentioning $g$).

Let $A$ be an $R$-$\Ptens$-algebra. We say that a short sequence of left $A$-$\Ptens$-modules
is {\em $R$-admissible} if it is split exact in $R\lmod$.
In the special case where $R=\CC$, we recover the standard definition of an admissible
(or $\CC$-split) sequence used in the homology theory of topological algebras
(see the original papers \cite{X70,T1,KV} and the monographs \cite{X1,EschmPut}).
The category $A\lmod$ together with the class of all short $R$-admissible sequences
is an exact category (in Quillen's sense); we denote this exact category by $(A,R)\lmod$.
Therefore the derived categories $\bD^*((A,R)\lmod)$ are defined, where
$*\in\{ +,-,b\}$; see \cite{DCU,Buhler} for details. The exact categories
$\rmod (A,R)$, $(A,R)\bimod (B,S)$, and their derived categories are defined similarly.

As was observed in \cite[Section 2]{Pir_qfree}, $(A,R)\lmod$ has enough projectives.
Hence each additive covariant (respectively, contravariant) functor from
$(A,R)\lmod$ to any exact category has a total left (respectively, right) derived functor
(see, e.g., \cite{DCU,Buhler}). In particular, let $\Vect$ denote the abelian
category of vector spaces, and let $\LCS$ denote the quasi-abelian category
of locally convex spaces (see \cite{Schneiders,Prosmans} for details about the latter).
Then for each
$Y\in \Ob((A,R)\lmod)$ we have the right derived functor
\begin{equation}
\label{RHom}
\RHom_A(-,Y)\colon\bD^-((A,R)\lmod)\to\bD^+(\Vect),
\end{equation}
and for each $X\in\Ob(\rmod (A,R))$ we have the left derived functor
\begin{equation}
\label{Lptens}
X\Lptens{A}(-)\colon \bD^-((A,R)\lmod)\to\bD^-(\LCS).
\end{equation}
There are some obvious variations of $\RHom_A$ and $\Lptens{A}$.
For example, if $B$ is an $S$-$\Ptens$-algebra and $X\in\Ob((B,S)\bimod (A,R))$,
then $X\Lptens{A}(-)$ may be viewed as a functor with values in $\bD^-((B,S)\lmod)$.

Given $X,Y\in\Ob((A,R)\lmod)$ and $n\in\Z_+$, we let
$\Ext^n_{A,R}(X,Y)=H^n(\RHom_A(X,Y))$, where $\RHom_A$ is given by \eqref{RHom}
and $H^n$ is the $n$th cohomology functor. Similarly, given
$X\in\Ob(\rmod (A,R))$ and $Y\in\Ob((A,R)\lmod)$, we let
$\Tor_n^{A,R}(X,Y)=H^{-n}(X\Lptens{A} Y)$, where $\Lptens{A}$ is given by \eqref{Lptens}.
Thus $\Ext^n_{A,R}(X,Y)$ is a vector space, while $\Tor_n^{A,R}(X,Y)$ is a
(not necessarily Hausdorff) locally convex space. The completion of
$\Tor_0^{A,R}(X,Y)$ is naturally isomorphic to $X\ptens{A} Y$
\cite[Section 2]{Pir_qfree}.

Suppose that $A$ is an $R$-$\Ptens$-algebra, $B$ is an $S$-$\Ptens$-algebra,
and $f\colon A\to B$ is an $R$-$S$-homomorphism. The ``restriction of scalars''
functor $f^\bullet\colon (B,S)\lmod \to (A,R)\lmod$ is obviously exact
(i.e., it takes admissible sequences to admissible sequences).
Hence $f^\bullet$ extends to a triangle functor
$\dD f^\bullet\colon \bD^*((B,S)\lmod)\to\bD^*((A,R)\lmod)$.

We will also need the notion of a homotopy limit (see, e.g., \cite{BN,Neeman}).
Suppose that $\cT$ is a triangulated category with countable products,
and let $(X_n,f_{n+1}\colon X_{n+1}\to X_n)_{n\in\N}$
be an inverse system in $\cT$. Define
\[
s\colon\prod_{n\in\N} X_n \to \prod_{n\in\N} X_n,\quad
s(x_1,x_2,\ldots )=(f_2(x_2),f_3(x_3),\ldots).
\]
An object $X$ of $\cT$ is a {\em homotopy limit} of $(X_n,f_n)$
if there exists a distinguished triangle
\[
X \to \prod_{n\in\N} X_n \xra{\id-s} \prod_{n\in\N} X_n \to X[1].
\]
In this case, we write $X=\holim X_n$. The homotopy limit exists and is
unique up to a noncanonical isomorphism.

Finally, recall that, if $\cA$ and $\cB$ are exact categories and
$F\colon\cA\to\cB$ is a covariant additive functor, then an object
$X$ of $\cA$ is (left) {\em $F$-acyclic} \cite[Section 15]{DCU} if the left derived functor $\dL F$
is defined at $X$ and if the canonical morphism $\dL F(X)\to F(X)$
is an isomorphism in $\bD(\cB)$.

\section{The results}

\begin{lemma}
\label{lemma:split_prod}
Let $\cA$ be an additive category with countable products, and let $f\colon X\to Y$
be a morphism in $\cA$. Define $i\colon X\to Y\times X^\N$ and
$t\colon Y\times X^\N \to Y\times X^\N$ by
\begin{align*}
i(x)&=(f(x),x,x,\ldots),\\
t(y,x_1,x_2,\dots)&=(y-f(x_1),x_1-x_2,x_2-x_3,\ldots).
\end{align*}
Then the sequence
\begin{equation}
\label{split_prod}
0 \to X \xra{i} Y\times X^\N \xra{t} Y\times X^\N \to 0
\end{equation}
is split exact in $\cA$.
\end{lemma}
\begin{proof}
Define $p\colon Y\times X^\N \to X$ and $q\colon Y\times X^\N \to Y\times X^\N$ by
\begin{align*}
p(y,x_1,x_2,\ldots)&=x_1,\\
q(y,x_1,x_2,\dots)&=\Bigl(y,0,-x_1,-(x_1+x_2),\ldots ,-\sum_{i=1}^{k-1} x_i,\ldots\Bigr).
\end{align*}
An elementary calculation shows that $pi=\id$, $tq=\id$, and $ip+qt=\id$.
\end{proof}

Let $\cA$ be an exact category with cokernels,
and let $X=(X^i, d^i\colon X^i\to X^{i+1})$ be a cochain complex in $\cA$.
For each $n\in\N$ let $\beta_n X$ denote the bounded complex
\[
0 \to \Coker d^{-n-1} \to X^{-n+1} \to \ldots \to X^{n-1} \to X^n \to 0
\]
concentrated in degrees from $-n$ to $n$. We have the obvious morphisms
$\beta_{n+1} X\to \beta_n X$, so $(\beta_n X)$ is an inverse system
in $\bC(\cA)$ and hence in $\bD(\cA)$. We clearly have $X=\varprojlim\beta_n X$ in $\bC(\cA)$.

Given $\cA$ as above, we say that $\cA$ {\em has exact countable products} if
countable products exist in $\cA$, and if the product of every countable family of
short admissible sequences in $\cA$ is admissible.

The following result is essentially contained in \cite[Prop. 6]{Meyer}.

\begin{lemma}
\label{lemma:holim}
If $\cA$ has exact countable products, then $X=\holim\beta_n X$ in $\bD(\cA)$.
\end{lemma}
\begin{proof}
Since countable products are exact in $\cA$, it follows that the
functor $\bC(\cA)\to\bD(\cA)$ preserves countable products \cite[Lemma 1.5]{BN}.
In particular, the product complex
$\prod_n \beta_n X$ taken in $\bC(\cA)$ (i.e., componentwise) is also a product
of the objects $(\beta_n X)$ in $\bD(\cA)$. To construct a triangle
\begin{equation}
\label{tri_beta}
X \to \prod_n \beta_n X \xra{1-s} \prod_n\beta_n X \to X[1]
\end{equation}
in $\bD(\cA)$, it suffices to find a sequence
$X \xra{i} \prod_n \beta_n X \xra{1-s} \prod_n\beta_n X$
in $\bC(\cA)$ such that, for each $k\in\Z$, the sequence
\begin{equation}
\label{level_k}
0 \to X^k \xra{i^k} \prod_n (\beta_n X)^k \xra{(1-s)^k} \prod_n (\beta_n X)^k \to 0
\end{equation}
is admissible \cite[Section 11]{DCU}. The obvious maps $X\to\beta_n X$ yield
$i\colon X\to\prod_n\beta_n X$. Now observe that \eqref{level_k} is a special case
of \eqref{split_prod} (for $k\ge 0$ we let $Y=0$, while for $k<0$ we let
$Y=\Coker d^{k-1}$ and $f\colon X^k\to\Coker d^{k-1}$ be the canonical map).
Applying Lemma~\ref{lemma:split_prod}, we see that \eqref{level_k} is split exact
and is {\em a fortiori} admissible in $\cA$.
\end{proof}

\begin{remark}
Since \eqref{level_k} is not only admissible, but is also split exact in $\cA$, it follows from
\cite[Section 6]{DCU} that \eqref{tri_beta} is a triangle not only in $\bD(\cA)$, but
also in the homotopy category $\bH(\cA)$. Thus $X=\holim\beta_n X$ in $\bH(\cA)$
as well (cf. \cite[Prop. 6]{Meyer}).
\end{remark}

\begin{lemma}
\label{lemma:semi_Yoneda}
Let $\cA$ be an exact category with cokernels and with exact countable products,
and let $X\to Y$ be a morphism in $\bD(\cA)$. Suppose that for each $Z\in\Ob(\bD^b(\cA))$
the induced map $\Hom_{\bD(\cA)}(Y,Z)\to \Hom_{\bD(\cA)}(X,Z)$ is a bijection.
Then $X\to Y$ is an isomorphism.
\end{lemma}
\begin{proof}
By Yoneda's lemma, it suffices to show that
$\Hom_{\bD(\cA)}(Y,Z)\to \Hom_{\bD(\cA)}(X,Z)$ is a bijection for each $Z\in\Ob(\bD(\cA))$.
Applying Lemma~\ref{lemma:holim}, we see that $Z=\holim\beta_n Z$ in $\bD(\cA)$.
We have a commutative diagram of abelian groups
\[
\xymatrix@C-10pt{
0  \ar[r]  & \varprojlim^1 \Hom_{\bD(\cA)}(Y,\beta_n Z[-1]) \ar[r] \ar[d]
& \Hom_{\bD(\cA)}(Y,Z)  \ar[r] \ar[d]
& \varprojlim \Hom_{\bD(\cA)}(Y,\beta_n Z) \ar[r] \ar[d] & 0\\
0  \ar[r]  & \varprojlim^1 \Hom_{\bD(\cA)}(X,\beta_n Z[-1]) \ar[r]
& \Hom_{\bD(\cA)}(X,Z)  \ar[r]
& \varprojlim \Hom_{\bD(\cA)}(X,\beta_n Z) \ar[r] & 0
}
\]
with exact rows (see \cite[Lemma 15.85.3]{Stacks}).
Since the left and the right vertical arrows are isomorphisms,
so is the middle vertical arrow. This completes the proof.
\end{proof}

Our next lemma is a variation of a well-known fact of category theory
(see, e.g., \cite[IV.3, Theorem 1]{ML_work} or \cite[Prop. 3.4.1]{Borceux}).

\begin{lemma}
\label{lemma:emb}
Let $\cA$ and $\cB$ be categories, and let $F\colon\cA\to\cB$ and
$G\colon\cB\to\cA$ be an adjoint pair of functors, with $F$ left adjoint to $G$.
Let $\cB_0$ be a full subcategory of $\cB$, and let $G_0\colon\cB_0\to\cA$
denote the restriction of $G$ to $\cB_0$.
Suppose that $\cB_0$ has the following property:
\begin{mycompactenum}
\item[$\mathrm{(*)}$]
A morphism $X\to Y$ in $\cB$ is an isomorphism if and only if
the induced map $\Hom_{\cB}(Y,Z)\to\Hom_{\cB}(X,Z)$ is bijective for all $Z\in\Ob(\cB_0)$.
\end{mycompactenum}
Then the following conditions are equivalent:
\begin{mycompactenum}
\item
$G_0$ is fully faithful;
\item
for each $X\in\Ob(\cB_0)$, the canonical morphism
$\eps_X\colon FG(X)\to X$ is an isomorphism.
\end{mycompactenum}
\end{lemma}
\begin{proof}
For each $X,Y\in\Ob(\cB)$ we have the following commutative diagram:
\[
\xymatrix@C+10pt{
\Hom_\cB(X,Y) \ar[r]^(.45){G_{X,Y}} \ar@/_30pt/[rr]_{\Hom(\eps_X,Y)}
& \Hom_\cA(GX,GY) \ar[r]_{\displaystyle{\sim}}^{\text{adj.}}
& \Hom_\cB(FGX,Y)
}
\]
(where ``adj.'' is the adjunction isomorphism). Hence (i) holds if and only if
$\Hom(\eps_X,Y)$ is bijective for all $X,Y\in\Ob(\cB_0)$. By $\mathrm{(*)}$, this is equivalent to (ii).
\end{proof}

We now come to the following corrected version of \cite[Theorem 6.3]{Pir_qfree}.

\begin{theorem}
\label{thm:left_hom_epi}
Let $f\colon A\to B$ be an $R$-$S$-homomorphism from an $R$-$\Ptens$-algebra $A$
to an $S$-$\Ptens$-algebra $B$. The following conditions are equivalent:
\begin{mycompactenum}
\item
$\dD f^\bullet\colon\bD^b((B,S)\lmod)\to\bD^b((A,R)\lmod)$
is fully faithful;
\item
for every $X,Y\in\Ob((B,S)\lmod)$ and every $n\in\Z_+$
the canonical morphism
\[
\Ext^n_{B,S}(X,Y)\to\Ext^n_{A,R}(X,Y)
\]
is bijective;
\item
for every $X\in\Ob(\bD^b((B,S)\lmod))$ the canonical morphism
$$
B\Lptens{A}\dD f^\bullet(X)\to X
$$
is an isomorphism in $\bD^-((B,S)\lmod)$;
\item
$f$ is an epimorphism of $\Ptens$-algebras, and, moreover,
every $X\in\Ob((B,S)\lmod)$ viewed as an object of $(A,R)\lmod$ is acyclic relative to the functor
$$
B\ptens{A}(-)\colon (A,R)\lmod\to (B,S)\lmod.
$$
\end{mycompactenum}
\end{theorem}
\begin{proof}
$\mathrm{(i)}\Longleftrightarrow\mathrm{(ii)}$.
This is a special case of \cite[Lemma 6.2]{Pir_qfree}.

$\mathrm{(i)}\Longleftrightarrow\mathrm{(iii)}$.
Let $\cA=\bD^-((A,R)\lmod)$, $\cB=\bD^-((B,S)\lmod)$, and $\cB_0=\bD^b((B,S)\lmod)$.
By Lemma~\ref{lemma:semi_Yoneda}, $\cB$ and $\cB_0$ satisfy condition $\mathrm{(*)}$
of Lemma~\ref{lemma:emb}. Let now
\[
F=B\Lptens{A}(-)\colon \cA \to \cB,\qquad
G=\dD f^\bullet\colon \cB \to \cA.
\]
By \cite[Lemma 13.6]{DCU}, $(F,G)$ is an adjoint pair, with $F$ left adjoint to $G$.
Applying now Lemma~\ref{lemma:emb}, we see that
$\mathrm{(i)}\Longleftrightarrow\mathrm{(iii)}$.

$\mathrm{(iii)}\Longrightarrow\mathrm{(iv)}$.
By (iii), for every $X\in\Ob((B,S)\lmod)$ the morphism
$B\Lptens{A} X\to X$ is an isomorphism in $\bD^-((B,S)\lmod)$ and
therefore in $\bD^-(\LCS)$. Applying the functor $H^0$, we see that the
composition of the canonical morphisms
$$
\Tor_0^{A,R}(B,X)\to B\ptens{A} X \to X
$$
is an isomorphism in $\LCS$.
Since the first morphism in this composition is the canonical map from a
locally convex space to its completion
(see \cite[Section 2]{Pir_qfree}), we conclude that both morphisms are isomorphisms in $\LCS$.
Hence $B\ptens{A} X\to X$ is an isomorphism in $B\lmod$. By composing
the isomorphisms $B\Lptens{A} X\cong X$ and $X\cong B\ptens{A} X$,
we conclude that $X$ is $B\ptens{A}(-)$-acyclic.
Finally, the canonical isomorphism $B\ptens{A} B\cong B$ means precisely
that $f$ is a $\Ptens$-algebra epimorphism \cite[Prop. 6.1]{Pir_qfree}.

$\mathrm{(iv)}\Longrightarrow\mathrm{(ii)}$.
Let $P_\bullet$ be a projective resolution of
$X$ in $(A,R)\lmod$. By (iv), $B\ptens{A} P_\bullet$ is a projective resolution of
$X$ in $(B,S)\lmod$. Therefore
\[
\Ext^n_{A,R}(X,Y) \cong H^n(\Hom_A(P_\bullet,Y)) \cong
H^n(\Hom_B(B\ptens{A}P_\bullet,Y)) \cong
\Ext^n_{B,S}(X,Y). \qedhere
\]
\end{proof}

\begin{definition}
If $f\colon A\to B$ satisfies the conditions of Theorem~\ref{thm:left_hom_epi}, then
we say that $f$ is a {\em left relative homological epimorphism}.
We say that $f$ is a {\em right relative homological epimorphism}
if $f\colon A^\op\to B^\op$ is a left relative homological epimorphism.
\end{definition}

\begin{remark}
In \cite[Theorem 6.3]{Pir_qfree}, which is the original version of Theorem~\ref{thm:left_hom_epi},
the morphism $B\Lptens{A}\dD f^\bullet(X)\to X$ (see (iii)) is claimed to be
a morphism in $\bD^b((B,S)\lmod)$. However, $B\Lptens{A}(-)$ does not take
$\bD^b((A,R)\lmod)$ to $\bD^b((B,S)\lmod)$ {\em a priori} (unless we impose some homological
finiteness conditions on $A$ or $B$). This resulted in a gap
in the proof of equivalence $\mathrm{(i)}\Longleftrightarrow\mathrm{(iii)}$,
because the functors $B\Lptens{A}(-)$ and $\dD f^\bullet$ cannot be viewed as
adjoint functors between the respective bounded derived categories, and so
\cite[IV.3, Theorem 1]{ML_work} does not apply. That is why we had to use
Lemmas~\ref{lemma:semi_Yoneda} and~\ref{lemma:emb} instead.
\end{remark}


\begin{thebibliography}{99}
\bibitem{AP}
Aristov, O. Yu. and Pirkovskii, A. Yu.
{\em Open embeddings and pseudoflat epimorphisms}.
J. Math. Anal. Appl. 485 (2020), no.~2, 123817.
\bibitem{BBK}
Ben-Bassat, O. and Kremnizer, K.
{\em Non-Archimedean analytic geometry as relative algebraic geometry}.
Ann. Fac. Sci. Toulouse Math. (6) 26 (2017), no.~1, 49--126.
\bibitem{BN}
B\"okstedt, M. and Neeman, A.
{\em Homotopy limits in triangulated categories}.
Compositio Math. 86 (1993), no.~2, 209--234.
\bibitem{Borceux}
Borceux, F.
{\em Handbook of categorical algebra. 1. Basic category theory}.
Cambridge University Press, Cambridge, 1994.
\bibitem{Buhler}
B\"uhler, Th.
{\em Exact categories}.
Expo. Math. 28 (2010), no.~1, 1--69.
\bibitem{Dicks}
Dicks, W.
{\em Mayer-Vietoris presentations over colimits of rings}.
Proc. London Math. Soc. (3) \textbf{34} (1977), no.~3, 557--576.
\bibitem{EschmPut}
Eschmeier, J. and Putinar, M.
{\em Spectral decompositions and analytic sheaves}.
London Mathematical Society Monographs. New Series, 10.
Oxford Science Publications. The Clarendon Press, Oxford University Press, New York, 1996.
\bibitem{GL}
Geigle, W. and Lenzing, H.
{\em Perpendicular categories with applications to representations and sheaves}.
J. Algebra \textbf{144} (1991), no.~2, 273--343.
\bibitem{X70}
Helemskii, A. Ya.
{\em The homological dimension of normed modules over Banach algebras}
(Russian). Mat. Sb. (N.S.) \textbf{81 (123)} (1970), 430--444.
\bibitem{X1}
Helemskii, A. Ya.
{\em The homology of Banach and topological algebras}.
Mathematics and its Applications (Soviet Series), 41.
Kluwer Academic Publishers Group, Dordrecht, 1989.
\bibitem{Hoch_rel}
Hochschild, G.
{\em Relative homological algebra}.
Trans. Amer. Math. Soc. 82 (1956), 246--269.
\bibitem{DCU}
Keller, B.
{\em Derived categories and their uses}.
Handbook of algebra, Vol.~1, 671--701,
Elsevier/North-Holland, Amsterdam, 1996.
\bibitem{KV}
Kiehl, R. and Verdier, J. L.
{\itshape Ein einfacher Beweis des Koh\"arenzsatzes von Grauert},
Math. Ann. \textbf{195} (1971), 24--50.
\bibitem{ML_work}
MacLane, S.
{\em Categories for the working mathematician}.
Springer-Verlag, New York-Berlin, 1971.
\bibitem{Meyer}
Meyer, R. {\em Embeddings of derived categories of bornological modules},
arXiv.org:math.FA/0410596.
\bibitem{Neeman}
Neeman, A.
{\em Triangulated categories}.
Princeton University Press, Princeton, NJ, 2001.
\bibitem{NR}
Neeman, A. and Ranicki, A. {\em Noncommutative localization and chain
complexes. I. Algebraic $K$- and $L$-theory}, arXiv.org:math.RA/0109118.
\bibitem{NS}
Nicol\'as, P. and Saor\'in, M.
{\em Parametrizing recollement data for triangulated categories}.
J. Algebra 322 (2009), no.~4, 1220--1250.
\bibitem{Pauk}
Pauksztello, D.
{\em Homological epimorphisms of differential graded algebras}.
Comm. Algebra 37 (2009), no.~7, 2337--2350.
\bibitem{Pir_qfree}
Pirkovskii, A. Yu.
{\em Arens-Michael envelopes, homological epimorphisms, and relatively
quasi-free algebras}. (Russian). Tr. Mosk. Mat. Obs. \textbf{69} (2008), 34--125;
translation in Trans. Moscow Math. Soc. 2008, 27--104.
\bibitem{Prosmans}
Prosmans, F.
{\em Derived categories for Functional Analysis}.
Publ. Res. Inst. Math. Sci. \textbf{36} (2000), no.~1, 19--83.
\bibitem{Schneiders}
Schneiders, J.-P.
{\em Quasi-abelian categories and sheaves}.
M\'em. Soc. Math. Fr. (N.S.) 1999, no.~76.
\bibitem{T1}
Taylor, J. L.
{\itshape Homology and cohomology for topological algebras},
Adv. Math. \textbf{9} (1972), 137--182.
\bibitem{T2}
Taylor, J. L.
{\itshape A general framework for a multi-operator functional calculus},
Adv. Math. \textbf{9} (1972), 183--252.
\bibitem{Stacks}
The Stacks Project, J.~A.~de~Jong (Editor).
{\ttfamily http://stacks.math.columbia.edu.}
\end{thebibliography}
\end{document}